\documentclass[11pt]{amsart}

\usepackage{amsrefs}

\usepackage{xcolor}

\usepackage{amsmath,amssymb}

\usepackage{graphicx}
\usepackage{mathrsfs}

\usepackage{mathtools}
\numberwithin{equation}{section}

\newtheorem{thm}{Theorem}[section]
\newtheorem{prop}[thm]{Proposition}

\newtheorem{cor}[thm]{Corollary}

\theoremstyle{definition}
\newtheorem{definition}[thm]{Definition}

\theoremstyle{remark}
\newtheorem{remark}[thm]{Remark}

\usepackage{dsfont}
\allowdisplaybreaks

\title[On a new equivalence relation for matrix valued orthogonal polynomials]{ON a new equivalence relation for\\ matrix valued orthogonal polynomials}
\author{Ignacio Bono Parisi}
\address{CIEM-FaMAF, Universidad Nacional de Córdoba, CP 5000, Córdoba, Argentina}
\email{ignacio.bono@unc.edu.ar}
\author{Inés Pacharoni}
\address{CIEM-FaMAF, Universidad Nacional de Córdoba, CP 5000, Córdoba, Argentina}
\email{ines.pacharoni@unc.edu.ar}
\author{Ignacio Zurrián}
\address{Departamento de Matemática Aplicada II, Universidad de Sevilla, Seville 41011}
\email{ignacio.zurrian@fulbrightmail.org}

\subjclass[2020]{33C45, 42C05, 33C47}

\keywords{Matrix-valued orthogonal polynomials, Darboux transformations, bispectrality.}

\begin{document}
\begin{abstract} 
In this work, we give some criteria that allow us to decide when two sequences of matrix-valued orthogonal polynomials are related via a Darboux transformation and how to build such a transformation explicitly. In particular, they allow us to see when and how any given sequence of polynomials is Darboux-related to classic orthogonal polynomials. We also explore the notion of Darboux-irreducibility and study some sequences that are not a Darboux transformation of classical orthogonal polynomials.
\end{abstract} 
\maketitle
\section{Introduction}

The issue of matrix-valued orthogonal polynomials (MVOP) satisfying a second-order differential equation was addressed for the first time in \cite{D97}. Since the first examples were built, from matrix spherical functions \cites{GPT01, GPT02, GPT03a} and from direct methods \cite{DG04},  this has been an area of interest where coexist different branches of mathematics such as representation theory, approximation theory, harmonic analysis, operator theory, special functions, etc.

More recently, in the very important work \cite{CY22}, the authors proved that every sequence of MVOP that satisfies a second-order differential equation and a certain condition must be a Darboux transformation of classical orthogonal polynomials. This condition is natural enough to encompass almost all the known cases. In \cite{BP23-1}, it was proven that an example from \cite{DG07} did not fit this condition, and later in \cite{BP24-1}, infinite families of examples of arbitrary size that do not fit it either are constructed. 

The above-mentioned condition requires the study of the corresponding algebra of operators $\mathcal D(W)$ and to verify whether its module rank coincides with the size of the matrices. 

One of the most important results in this paper is that we give a simple condition for two sequences of MVOP to be related by a Darboux transformation, together with an explicit construction of such a transformation. 
In particular, this allows us to understand in more direct manner when a sequence of MVOP is a Darboux transformation of (a direct sum of) smaller MVOPs. It is worth to observe that our results allow to study directly the link beetwen two sequences of MVOPs considering only the polynomials involved, without a need for the explicit weights (Example 3 in Section \ref{E3}), or considering only the weights, with no need of the expressions of polynomials involved (Example 1 in Section \ref{E1}).
In any case, it is not necessary to make a full study of the corresponding algebras of operators.

In Section \ref{sec-MOP}, we give preliminary notions related to MVOP and the algebra $\mathcal D(W)$. In Section \ref{Darboux-sect}, we recall the notion of Darboux transformation and give an equivalent and simpler condition for a sequence of MVOP to be Darboux transformation of another sequence of MVOP (Theorem \ref{darboux-thm}); we also study the modules $\mathcal D(W,\tilde W)$ (Theorem \ref{op mod}), which later in the paper are a tool to decide directly when there is a Darboux transformation between two MVOPs. Also, we prove that a Darboux Transformation gives a relation of equivalence for MVOP. In Section \ref{S-DTcp}, we give results that allow us to verify whether a sequence of MVOP is Darboux-equivalent to a diagonal matrix of classical orthogonal polynomials and build explicitly such a transformation. We apply these results in two examples.
In Section \ref{S-DT}, we formulate the results for general sizes. 
This gives a criterion that decides when a sequence of MVOP is Darboux-reducible. We also apply these results in two examples.

We are convinced that these results will be useful by giving a method of building the explicit Darboux transformations relating known families of MVOPs with new ones, as well as by giving a new tool to generate explicit new examples.
Besides, they enable a different approach to MVOPs that are not Darboux transformations of classical orthogonal polynomials. Furthermore, given the generality of these results, they may be adapted to a broader type of orthogonal polynomials: (matrix and scalar-valued) Krall polynomials, exceptional polynomials, etc..

\section{Matrix valued orthogonal polynomials and the algebra $\mathcal D(W)$}\label{sec-MOP}

 Let $W=W(x)$ be a matrix weight of size $N$ on the real line, that is, a complex $N\times N$ matrix-valued smooth function on the possibly infinite interval $\mathcal{I}=(x_0,x_1)$ such that $W(x)$ is positive definite almost everywhere with finite moments of all orders. 
 Let $\operatorname{Mat}_N(\mathbb{C})$ be the algebra of all $N\times N$ complex matrices and let $\operatorname{Mat}_N(\mathbb{C}[x])$ be the algebra of polynomials in the indeterminate $x$ with coefficients in $\operatorname{Mat}_N(\mathbb{C})$. We consider the following Hermitian sesquilinear form in the linear space $\operatorname{Mat}_N(\mathbb{C}[x])$
\begin{equation*}
  \langle P,Q \rangle =  \langle P,Q \rangle_W = \int_{x_0}^{x_1} P(x) W(x) Q(x)^*\,dx,
\end{equation*}
where $*$ denotes the Hermitian conjugate.

With this sesquilinear form, one can construct sequences 
$\{P_n\}_{n\in\mathbb{N}_0}$ of matrix-valued orthogonal polynomials (MVOP) with respect to $W$. That is, each $P_n$ is a polynomial of degree $n$ with a nonsingular leading coefficient and $\langle P_n, P_m \rangle = 0$ when $n \neq m$.
It follows that two sequences   $\{P_{n}(x)\}_{n\in\mathbb{N}_{0}}$ and $\{R_{n}(x)\}_{n\in \mathbb{N}_{0}}$ of MVOPs with respect to $W$ are related by $R_{n}(x) = M_{n}P_{n}(x)$, for some sequence of nonsingular constant matrices $M_{n}$.
When the leading coefficient of a polynomial is the identity matrix $I$ we call it {\it monic}.
In particular, there is a unique sequence of monic orthogonal polynomials with respect to $W$.

By following a standard argument (see \cite{K49} or \cite{K71}) one shows that a sequence of orthogonal polynomials $\{P_n\}_{n\in\mathbb{N}_0}$ satisfy a three-term recursion relation
\begin{equation}\label{ttrr}
    P_n(x)x=A_{n}P_{n+1}(x) + B_{n}P_{n}(x)+ C_nP_{n-1}(x), \qquad n\in\mathbb{N}_0,
\end{equation}
where $P_{-1}=0$ and $A_{n}, B_n, C_n$ are matrices depending on $n$ and not on $x$.

In general, MVOPs may not satisfy a (non-trivial) differential equation, but the most interesting cases appear when they do and are called bispectral. In the scalar case, these bispectral polynomials are the classical polynomials, Krall polynomials, etc. 

Let us consider matrix differential operators
$  D=\sum_{i=0}^s \partial ^i F_i(x)$,  with $\partial=\frac{d}{dx} $,
acting on the right-hand side on a matrix-valued function $P$, namely $(P\cdot D)(x) = \sum_{i=0}^{s} \frac{d^{i}P}{dx^{i}}(x)F_{i}(x)$.

We denote the algebra of these operators with polynomial coefficients by $$\operatorname{Mat}_{N}(\Omega[x])=\Big\{D = \sum_{i=0}^{s} \partial^{i}F_{i}(x) \, : F_{i} \in \operatorname{Mat}_{N}(\mathbb{C}[x]) \Big \}.$$
\noindent 
More generally, when necessary, we will also consider $\operatorname{Mat}_{N}(\Omega[[x]])$, the set of all differential operators with coefficients in $\mathbb{C}[[x]]$, the ring of power series with coefficients in $\mathbb{C}$.

Given a weight matrix $W$, the algebra
\begin{equation}\label{algDW}
  \mathcal D(W)=\left\{D\in \operatorname{Mat}_{N}(\Omega[x]) :  \,P_n\cdot D=\Lambda_n(D) P_n \text{ for some } \Lambda_n(D)\in \operatorname{Mat}_N(\mathbb{C}), \text{ for all } n\in\mathbb{N}_0\right\}
\end{equation}
is introduced in \cite{GT07}, where $\{P_n\}_{n\in \mathbb{N}_0}$ is any sequence of MVOP with respect to $W$ and as usual $\mathbb N_0=\{0,1,\dots  \}$.

\begin{definition}
 The {\em formal adjoint} of an operator $D=\sum_{j=0}^{n} \partial^{j}F_{j}(x)$, denoted  by $D^*$, is the operator given by $D^*=\sum_{j=0}^{n} F_{j}^*(x) (-1)^j\partial^{j}.$ The {\em formal $W$-adjoint} of $ \mathfrak{D}\in \operatorname{Mat}_{N}(\Omega[x])$, or  the formal adjoint of $\mathfrak D$ with respect to $W(x)$  is the differential operator $\mathfrak{D}^{\dagger} \in \operatorname{Mat}_{N}(\Omega[[x]])$ defined
by
$$\mathfrak{D}^{\dagger}:= W(x)\mathfrak{D}^{\ast}W(x)^{-1},$$
where  $\mathfrak{D}^{\ast}$ is the formal adjoint of 
$\mathfrak D$. 

\smallskip

An operator $\mathfrak{D}\in \operatorname{Mat}_{N}(\Omega[x])$ is called {\em $W$-adjointable} if there exists 
$\tilde  {\mathfrak{D}} \in \operatorname{Mat}_{N}(\Omega[x])$, such that
$$\langle P\cdot \mathfrak{D},Q\rangle=\langle P,Q\cdot \tilde {\mathfrak{D}}\rangle,$$ for all $P,Q\in \operatorname{Mat}_N(\mathbb{C}[x])$. Then we say that the operator $\tilde {\mathfrak D}$ is the $W$-adjoint of $\mathfrak D $.
\end{definition}

\begin{definition}
We say that a differential operator $D\in \mathcal D(W)$ is $W$-{\em symmetric} if $\langle P\cdot D,Q\rangle=\langle P,Q\cdot D\rangle$, for all $P,Q\in \operatorname{Mat}_N(\mathbb{C}[x])$.
An operator $\mathfrak{D}\in \operatorname{Mat}_{N}(\Omega[x])$ is called {\em formally } $W${\em -symmetric} if $\mathfrak{D}^{\dagger} = \mathfrak{D}$.
\end{definition}

In \cite[Corollary 4.5]{GT07}, it is shown that the set $\mathcal S(W)$ of all $W$-symmetric operators in $\mathcal D(W)$ is a real form of the space $\mathcal D(W)$, i.e.
\begin{equation}\label{symop}
    \mathcal D(W)= \mathcal S (W)\oplus i \mathcal S (W),
\end{equation}
as real vector spaces. 

Notice that, using \eqref{symop} and \cite[Proposition 2.10]{GT07}, $\mathcal D(W)$ may also be regarded as the set of operators $D = \sum_{j=0}^{m} \partial^{j}F_{j}(x)$, with $F_{j}$ a polynomial of degree less than or equal to $j$, that are $W$-adjointable.

\begin{definition} We say that a matrix function $F$ is the direct sum of the matrix functions 
$F_1,F_2$, and write $F=F_1\oplus F_2$, if 
$$F(x)=\left(\begin{matrix}F_1(x)&{\bf 0} \\ {\bf0} & F_2(x) \end{matrix}\right).$$
Similarly, we define $F=F_1\oplus F_2\oplus\cdots\oplus F_j$.  
\end{definition}

Let us recall that for a given sequence of polynomials there may be more than one weight for which they are orthogonal on an infinite interval, i.e. \cite[Example 4.13]{TZ18}.
In this work, to avoid pathological weights and their technicalities, since we are only interested in the sequences of orthogonal polynomials, we will consider weights $W$ ``good enough". Namely, if $D\in \operatorname{Mat}_{N}(\Omega[x])$ and $A\in\operatorname{Mat}_{N}(\mathbb{C}[x])$ satisfy
$$\int_{x_0}^{x_1} P(x) A(x) W(x)  \,dx=\int_{x_0}^{x_1} P(x) (WD)(x)\,dx,\quad \text{ for all } P\in \operatorname{Mat}_{N}(\mathbb{C}[x]),$$
then we have $A(x) W(x)= (WD)(x)$.
 
In particular, for any $W$-adjointable operator in $\operatorname{Mat}_{N}(\Omega[x])$ the formal $W$-adjoint coincides with the $W$-adjoint. Moreover, since all operators in $\mathcal{D}(W)$ are $W$-adjointable, we have that $D \in \mathcal{D}(W)$ is $W$-symmetric if and only if it is formally $W$-symmetric.
The interested reader may consult \cite[Section 2.2]{CY22} for a more detailed account of the possible situations. The authors in that work consider an assumption (\cite[Assumption 2.21]{CY22}) that implies the ours. 
 
\section{The Darboux transformation}\label{Darboux-sect}

Throughout this section, $P_n=P_n(x)$ and $\tilde P_n=\tilde P_n(x)$ denote sequences of matrix-valued orthogonal polynomials with respect to some weights $W=W(x)$ and $\tilde W=\tilde W(x)$, respectively.

\begin{definition}
    We say that a linear differential operator $E$ is {\it degree-preserving} operator if, for any polynomial $f$ of degree $n$, the function $f\cdot E$ is again a polynomial of degree $n$, for all but finitely many $n\in \mathbb{N}_{0}$. 
\end{definition}

\begin{definition}\label{darboux-def}
We say that $\widetilde{P}_{n}$ is a {\em Darboux transformation} of ${P}_{n}$ if there exists a differential operator $D \in \mathcal{D}(W)$ that can be factorized as $D = \mathcal{V}\mathcal{N}$ with degree-preserving operators $\mathcal{V}, \, \mathcal{N} \in \operatorname{Mat}_{N}(\Omega[x])$, such that
    \begin{equation*}
            P_{n}(x) \cdot \mathcal{V}  = A_{n}\widetilde{P}_{n}(x),
    \end{equation*}
for a sequence of matrices $A_{n} \in \operatorname{Mat}_{N}(\mathbb{C})$, for all $n\in \mathbb N_0$. 
\end{definition}

Theorems \ref{darboux-thm} and \ref{dr} below allow us to characterize in a much simpler manner the situations in which there is a Darboux transformation between two sequences of MVOPs. 

\begin{thm}	\label{dr}
If there exists a differential operator $\mathcal{V}\in \operatorname{Mat}_{N}(\Omega[x])$ such that 
    \begin{equation}\label{darboux-1}
            P_{n}(x) \cdot \mathcal{V}  = A_{n}\tilde {P}_{n}(x), \quad \text{ for all } n\in \mathbb{N}_{0}, 
    \end{equation}
with $A_{n} \in \operatorname{Mat}_{N}(\mathbb{C})$ invertible for some $n$, then 
we have that  there exists a differential operator $\mathcal{N} \in \operatorname{Mat}_{N}(\Omega[x])$ such that 
    \begin{equation}\label{darboux-2}
\tilde             P_{n}(x) \cdot \mathcal{N}  = \tilde A_{n}{P}_{n}(x), \quad \text{ for all } n\in \mathbb{N}_{0}, 
    \end{equation}
with $\tilde A_{n} \in \operatorname{Mat}_{N}(\mathbb{C})$ invertible for some $n$. 
In particular, 
$P_n \cdot \mathcal{V}\mathcal{N}=A_n\tilde A_n P_n$ and 
$\tilde P_n \cdot \mathcal{N}\mathcal{V}=\tilde A_n A_n \tilde P_n$, i.e. $\mathcal{V}\mathcal{N}\in\mathcal D(W)$ and $\mathcal{N}\mathcal{V}\in\mathcal D(\tilde W)$.  Furthermore, $$ \mathcal N=\tilde W\mathcal V^*W^{-1}.$$
\end{thm}	
\begin{proof}
By hypothesis we have 
$$
\begin{pmatrix}
P_n&0\\
0&\tilde P_n
\end{pmatrix}
\cdot \begin{pmatrix}
0&\mathcal V\\
0&0
\end{pmatrix}=
\begin{pmatrix}
0&A_n\\
0&0
\end{pmatrix}
\begin{pmatrix}
P_n&0\\
0&\tilde P_n
\end{pmatrix}.
$$
Since $P_n\oplus \tilde P_n$ are the MVOP with respect to  $\hat W:=W\oplus\tilde W$, we have that $V:=\begin{pmatrix}
0&\mathcal V\\
0&0
\end{pmatrix}\in \mathcal D(\hat W)$ and  $\Lambda_n(V)=\begin{pmatrix}
0&A_n\\
0&0
\end{pmatrix}
$. 
Therefore, $V$ is $\hat W$-adjointable with $V^\dagger\in \mathcal D(\hat W)$. On the other hand, $V^\dagger=\hat  WV^*\hat  W ^{-1} $, then 
$$V^\dagger
=
\begin{pmatrix}
0&0\\
\mathcal N&0
\end{pmatrix}, \quad \text{ with }  \mathcal N=\tilde W\mathcal V^*W^{-1}.
$$
Let us call $\Lambda_n(V^{\dagger})$ the eigenvalues such that
\begin{equation}\label{1}
\begin{pmatrix}
P_n&0\\
0&\tilde P_n
\end{pmatrix} \cdot
V^{\dagger}
=
\begin{pmatrix}
P_n&0\\
0&\tilde P_n
\end{pmatrix} \cdot
\begin{pmatrix}
0&0\\
\mathcal N&0
\end{pmatrix}=
\Lambda_n(V^{\dagger})
\begin{pmatrix}
P_n&0\\
0&\tilde P_n
\end{pmatrix}.
\end{equation}
Then, we have
\begin{align*}
\Lambda_n(V^{\dagger})\left(\begin{smallmatrix}
\|P_n\|^2_W&0\\
0&\|\tilde P_n\|^2_{\tilde W}
\end{smallmatrix}\right)
=&
\left\langle 
\Lambda_n(V^{\dagger})
\left(\begin{smallmatrix}
P_n&0\\
0&\tilde P_n
\end{smallmatrix}\right)
,
\left(\begin{smallmatrix}
P_n&0\\
0&\tilde P_n
\end{smallmatrix}\right)\right\rangle_{\hat W}
&=&
\left\langle
\left(\begin{smallmatrix}
P_n&0\\
0&\tilde P_n
\end{smallmatrix}\right)
\cdot V^{\dagger},
\left(\begin{smallmatrix}
P_n&0\\
0&\tilde P_n
\end{smallmatrix}\right)\right\rangle_{\hat W}
\\
=&
\left\langle
\left(\begin{smallmatrix}
P_n&0\\
0&\tilde P_n
\end{smallmatrix}\right)
,
\left(\begin{smallmatrix}
P_n&0\\
0&\tilde P_n
\end{smallmatrix}\right) \cdot
V
\right\rangle_{\hat W}
&=&
\left\langle
\left(\begin{smallmatrix}
P_n&0\\
0&\tilde P_n
\end{smallmatrix}\right)
,\Lambda_n(V)
\left(\begin{smallmatrix}
P_n&0\\
0&\tilde P_n
\end{smallmatrix}\right)
\right\rangle_{\hat W}
\\
=&
\left(\begin{smallmatrix}
\|P_n\|^2_W&0\\
0&\|\tilde P_n\|^2_{\tilde W}
\end{smallmatrix}\right) \Lambda_n(V)^*
&=&
\left(\begin{smallmatrix}
\|P_n\|^2_W&0\\
0&\|\tilde P_n\|^2_{\tilde W}
\end{smallmatrix}\right)
\left(\begin{matrix}
0&0\\
A_n^*&0
\end{matrix}\right).
\end{align*}
It follows that $$ 
\Lambda_n(V^\dagger)=\begin{pmatrix}
0&0\\
\tilde A_n&0
\end{pmatrix}, 
\quad
\text{ with } \tilde A_n=\|\tilde P_n\|_{\tilde W}^2 A_n^* \| P_n\|_W^{-2}.$$
Hence, from \eqref{1}, we have proved that
$$\tilde P_n \cdot \mathcal N=\tilde A_n P_n,$$
with $\tilde A_n$ invertible for some $n$. 
Finally, by combining this with \eqref{darboux-1} we have
$$P_n \cdot\mathcal V\mathcal N=A_n\tilde A_nP_n,\quad \text{ and }\quad
\tilde P_n \cdot \mathcal{N}\mathcal{V}=\tilde A_n A_n \tilde P_n.$$ This completes the proof.
\end{proof}

\begin{thm}\label{darboux-thm} 
If there exists a differential operator $\mathcal{V}\in \operatorname{Mat}_{N}(\Omega[x])$ such that 
    \begin{equation*}
            P_{n}(x) \cdot \mathcal{V}  = A_{n}\tilde {P}_{n}(x), \quad \text{ for all } n\in \mathbb{N}_{0}, 
    \end{equation*}
with $A_{n} \in \operatorname{Mat}_{N}(\mathbb{C})$ invertible for some $n$, then $\tilde {P}$ is a {Darboux transformation} of $P$. The converse is also true. 
\end{thm}
\begin{proof}
Firstly, let us observe that if we write $P_n=M_n Q_n $ and $\tilde P_n=\tilde M_n \tilde Q_n $ with $Q_n$ and $\tilde Q_n$ monic orthogonal polynomials, it follows that $Q_n \cdot \mathcal V= M_n^{-1}A_n\tilde M_n \tilde Q_n$. Since the coefficients of $\mathcal V$ are polynomials on $x$ it is easy to see that $M_n^{-1}A_n\tilde M_n$ is a polynomial on $n$. Hence, given that $A_n$ is invertible for some $n$ we have that $M_n^{-1}A_n\tilde M_n$ is invertible for some $n$, and since it is a polynomial we have that $M_n^{-1}A_n\tilde M_n$ is invertible for all but finitely many $n$. Then, $A_n$ is invertible for all but finitely many $n$. Besides, by applying Theorem \ref{dr} we know that there exists a differential operator $\mathcal{N} \in \operatorname{Mat}_{N}(\Omega[x])$ such that $\tilde             P_{n}(x) \cdot \mathcal{N}  = \tilde A_{n}{P}_{n}(x)$
with $\tilde A_{n}$ invertible for some $n$. Analogously, $\tilde A_n$ is invertible for all but finitely many $n$ as well.

Also from Theorem \ref{dr}, we have that $P_n \cdot \mathcal V \mathcal N=A_n\tilde A_n P_n$, then the operator $D=\mathcal V \mathcal N$ belongs to the algebra $\mathcal D(W)$. On the other hand, the operators $\mathcal V $ and $\mathcal N$ are degree-preserving because of \eqref{darboux-1} and \eqref{darboux-2}, together with the fact that $A_{n}$ and $\tilde A_{n}$ are invertible for all but finitely many $n$. Hence, $\tilde {P}$ is a {Darboux transformation} of $P$. The converse is immediate. 
\end{proof}

\begin{remark}
Notice that we did not require in the hypothesis of Theorems \ref{darboux-thm} or \ref{dr} that the functions $P_n(x)$ or $\tilde P_n(x)$ satisfy any differential (non-trivial) equation, this property is a consequence (when $\mathcal V$ is of order positive). 
\end{remark}

As a consequence of Theorem \ref{darboux-thm} and Theorem \ref{dr}, we have the following result.
\begin{cor}	
  $\tilde {P}$ is a {Darboux transformation} of $P$ if and only if $ {P}$ is a {Darboux transformation} of $\tilde P$.
\end{cor}

The Darboux transformation defines an equivalence relation between matrix-valued orthogonal polynomials:
\begin{equation}\label{equivrelat}
P_n \sim \tilde {P_n}  \qquad \text{ if and only if }  \qquad \tilde {P_n}  \text{ is a Darboux transformation of $P_n$}. 
\end{equation}

There is an older notion of equivalence for MVOP, stating that two weights $W$ and $\tilde W$ are \emph{similar} if there exists a nonsingular constant matrix $M$ such that ${W}(x) = M\tilde W(x)M^{\ast}$. This implies that for any two similar MVOPs $P_n$ and $\tilde P_n$ associated to $W$ and $\tilde W$, respectively, one has that $ P_n=A_n \tilde P_n M^{-1}$, with $A_n$ invertible, for all $n\in\mathbb N_0$. 
Then, from Theorem \ref{darboux-thm} we have that if $P_n$ and $\tilde P_n$ are similar, then they equivalent via a Darboux transformation (with $\mathcal V=M$), showing that this newer equivalence generalizes the old one. As a consequence, we can now consider a new notion of reducibility.

\begin{definition}
We say that a sequence $\{P_n\}_{n\in \mathbb{N}_{0}}$ of $N\times N$ matrix-valued orthogonal polynomials is  \emph{Darboux-reducible} if $$P_n \sim \left(\begin{matrix}Q_n& 0\\0&R_n\end{matrix}\right),$$
where $\{Q_n\}_{n\in \mathbb{N}_{0}}$ and  $\{R_n\}_{n\in \mathbb{N}_{0}}$ are sequences of  matrix-valued orthogonal polynomials of sizes $a$ and $b$, respectively, with $N=a+b$. We say that such a sequence $\{P_n\}_{n\in \mathbb{N}_{0}}$ is \emph{ Darboux-irreducible} if it is not Darboux-reducible.
\end{definition}

\begin{prop} \label{alg debil}
    Let  $\tilde {P_n}$ be  a Darboux transformation of $P_n$. Let   $\mathcal{V}, \, \mathcal{N} $ be the operators in \eqref{darboux-1} and \eqref{darboux-2}. Then
    \begin{enumerate}
    \item[i.] The differential operators $\left(\begin{smallmatrix} 0 && \mathcal{V} \\ 0 && 0 \end{smallmatrix}\right)$ and $  \left(\begin{smallmatrix} 0 && 0 \\ \mathcal{N} && 0 \end{smallmatrix}\right) $ belong to  $\mathcal{D}(W \oplus \tilde {W})$. 
    
        \item [ii.]
    The algebras satisfy
    $$\mathcal{V}\mathcal{D}(\tilde {W})\mathcal{N} \subseteq \mathcal{D}(W) \quad \text{ and } \quad \mathcal{N} \mathcal{D}(W)\mathcal{V} \subseteq \mathcal{D}(\tilde {W}).$$
    \end{enumerate}
\end{prop}
\begin{proof}
The item (i) follows straightforwardly from observing that if $\mathcal N$ satisfies \eqref{darboux-2} then 
$$
\begin{pmatrix}
P_n&0\\
0&\tilde P_n
\end{pmatrix}
\cdot \begin{pmatrix}
0&0\\
\mathcal N&0
\end{pmatrix}=
\begin{pmatrix}
0&0\\
\tilde A_n&0
\end{pmatrix}
\begin{pmatrix}
P_n&0\\
0&\tilde P_n
\end{pmatrix},
$$
hence, $\left(\begin{smallmatrix} 0 && 0 \\ \mathcal{N} && 0 \end{smallmatrix}\right) $ belongs to  $\mathcal{D}(W \oplus \tilde {W})$. For $\left(\begin{smallmatrix} 0 && \mathcal{V} \\ 0 && 0 \end{smallmatrix}\right)$ it is completely analogous.

It is worth to observe that for any $D$ such that $P_n \cdot D=\Lambda_n(D)P_n$ and $\mathcal V, \, \mathcal N$ satisfying \eqref{darboux-1} and \eqref{darboux-2}, one has 
$$
\tilde P_n \cdot \mathcal N D \mathcal V =  A_n \Lambda_n(D)\tilde A_n \tilde P_n.
$$
Similarly, if $\tilde D$ is such that $\tilde P_n \cdot \tilde D=\Lambda_n(\tilde D)\tilde P_n$, then
$$
P_n \cdot \mathcal V \tilde  D \mathcal N =  \tilde  A_n \Lambda_n(\tilde D) A_n P_n.
$$
This finishes the proof.
\end{proof}

\subsection{The modules $\mathcal D(W,\tilde W)$ and the Darboux-transformations}
Notice that in Definition \ref{darboux-def} there are no mentions of the sizes of $P_n$ and $\tilde P_n$, but \eqref{darboux-1} implies that they need to have the same size to be Darboux-related. In that regard, for the study of relationships between matrix-valued orthogonal polynomials of different sizes, and the results in the following sections, it is useful to recover the following notion from \cite{TZ18}.

\begin{definition} \label{modulos Ww}
Let the sizes of $P_n$ and $\tilde P_n$ be $N_{1} $ and $N_{2}$, respectively.
 We consider 
    \begin{equation*}
            \mathcal{D}(W,\tilde  W)  = \left\{ \mathcal T \in \operatorname{Mat}_{N_{1} \times N_{2}}(\Omega[x]) \, : \, P_{n}(x) \cdot \mathcal T = A_n \tilde  P_{n}(x), \text{ with } A_n \in \operatorname{Mat}_{N_{1}\times N_{2}}(\mathbb{C}) \right \}.
    \end{equation*}
\end{definition}

Observe that $\mathcal D(W, \tilde  W)$
is a left $\mathcal D(W) $-module and a right $\mathcal D(\tilde  W)$-module.  Besides 
\begin{equation} \label{D(W,w)D(w,W)}
     D(W,\tilde  W) \mathcal D(\tilde  W,W) \subset \mathcal D(W) \qquad \text{ and } \qquad \mathcal D(W, W) = \mathcal D(W).
\end{equation}
For the cases when both $W$ and $\tilde W$ are scalar-valued, these modules were studied in \cite{BP23-2}.

We have the following straightforward proposition. 
\begin{prop}\label{basicprop}
    Let $W$ and $\tilde W$ be weight matrices of size $N_1$ and $N_2$, respectively. 
    \begin{enumerate}
        \item [i.] If $ \mathcal T\in \mathcal D(  W,\tilde W) $ then $T=\left(\begin{smallmatrix} 0 & \mathcal T \\ 0 & 0 \end{smallmatrix}\right)\in \mathcal D(W\oplus \tilde W)$.
    
   \item [ii.] If $ \mathcal S\in \mathcal D( \tilde W, W) $ then $S=\left(\begin{smallmatrix} 0 & 0 \\ \mathcal S &0\end{smallmatrix}\right)\in \mathcal D(W\oplus \tilde W)$.
   
   \item [iii.] If $\mathcal T\in \mathcal D(W, \tilde W)  $ then $\mathcal T^\dagger:= \tilde W(x) \mathcal T^* W^{-1}(x)\in \mathcal D(\tilde W, W).$
    \end{enumerate}
\end{prop}
\begin{proof}
 If we proceed as in the beginning of the proof of Theorem \ref{dr} we have that if $\mathcal T\in \mathcal D(W, \tilde W)$ then $T=\left(\begin{smallmatrix} 0 & \mathcal{T} \\ 0 & 0 \end{smallmatrix}\right)\in \mathcal D(W\oplus \tilde W)$. Analogously, it is easy to see that $\left(\begin{smallmatrix} 0 & 0 \\  \mathcal{S}& 0 \end{smallmatrix}\right)\in \mathcal D(W\oplus \tilde W)$. 

 To prove (iii), let us recall that since $T\in \mathcal D(W\oplus \tilde W)$ we have that $ T$ is $(W\oplus \tilde W)$-adjointable and its adjoint is 
$T^\dagger=\left(\begin{smallmatrix} 0 & 0 \\ \tilde W\mathcal{T}^* W^{-1} & 0 \end{smallmatrix}\right)\in \mathcal D(W\oplus \tilde W)$. Thus we have 
 $\mathcal T^\dagger= \tilde W \mathcal T^* W^{-1} \in \mathcal D(\tilde W, W).$
\end{proof}

\smallskip

If we consider a weight matrix $W$ of size $N$ of the form 
$$W = W_1\oplus  W_2 \oplus \cdots \oplus W_r,$$ where 
 $W_1$, $W_2, \dots , W_r$ are weight matrices of size $N_1$, $N_2, \dots N_r$, respectively. 
A sequence  of orthogonal polynomials with respect to $W$ is given by 
$$P_n(x)=
P_{1,n}(x)\oplus  P_{2,n}(x) \oplus \cdots \oplus P_{r,n}(x)
, $$
where  $\{P_{j,n}\}_{n\in \mathbb{N}_{0}}$  
is a  sequence of orthogonal polynomials with respect  to $W_j$, for $j=1,\dots,r$. 
A differential operator $D \in  \mathcal D(W)$  can be written as a block-matrix
$$ D=\begin{pmatrix}
    D_{1,1} & D_{1,2} & \cdots & D_{1,r}\\ D_{2,1} & D_{2,2}
     && D_{2,r} \\  && \ddots  &&
\end{pmatrix}$$
where $D_{i,j}$ is a 
differential operator of size $N_i \times  N_j$, and
$D_{i,j} \in  \mathcal D(W_i, W_j)$. In particular, we have that  
 $$ \mathcal D(W_1\oplus  W_2 \oplus \cdots \oplus W_r)=\begin{pmatrix}
     \mathcal D(W_1) & \mathcal D(W_1,W_2) & \cdots  & \mathcal D(W_1, W_r) \\  \mathcal D(W_2,W_1) & \mathcal D(W_2) & \cdots & \mathcal D(W_2, W_r) \\ & \ddots & & \\ && \ddots &  \\ 
      \mathcal D(W_r, W_1) & \mathcal D(W_r,W_2) & \cdots  & \mathcal D(W_r)
 \end{pmatrix}.$$

\bigskip

The following theorem will be instrumental for a better understanding of the applications of our results to the examples below. The statement may remind the reader of the symmetry equations from \cite[p.~357]{GPT03a} or \cite{DG04}.

\begin{thm} \label{op mod}
    Let $W$ and $\tilde{W}$ be weight matrices of size $N_{1}$ and $N_{2}$ respectively, supported on the same interval $(x_{0},x_{1})$. Let $\mathcal{V} = \partial^{2}F_{2}(x) + \partial F_{1}(x) + F_{0}(x)$ be a differential operator with $F_{i} \in \mathbb{C}^{N_{1}\times N_{2}}[x]$  and $\deg(F_{i})\leq i$. The differential operator $\mathcal{V}$ belongs to the module $\mathcal{D}(W,\tilde{W})$ if and only if there exist matrix-valued polynomials $G_{i} \in \mathbb{C}^{N_{2} \times N_{1}}[x]$, $i=0,1,2$, with $\deg(G_{i})\leq i$, such that 
    \begin{equation}\label{eq_sym}
        \begin{split}
        G_{2}(x)W(x) & = \tilde{W}(x)F_{2}(x)^{\ast}, \\
        G_{1}(x)W(x) & = 2(\tilde{W}(x)F_{2}(x)^{\ast})'-\tilde{W}(x)F_{1}(x)^{\ast}, \\
        G_{0}(x)W(x) & = (\tilde{W}(x)F_{2}(x)^{\ast})''-(\tilde{W}(x)F_{1}(x)^{\ast})'+\tilde{W}(x)F_{0}(x)^{\ast},
        \end{split}
    \end{equation}
    with the boundary conditions
    \begin{equation}\label{cond_sym}\lim_{x\to x_{0},x_{1}}F_{2}(x)\tilde{W}(x) = 0, \quad \lim_{x\to x_{0}, x_{1}} (F_{1}(x)\tilde{W}(x)-W(x)G_{1}(x)^{\ast})=0.\end{equation}
    Moreover, the differential operator $\mathcal{N} = \partial^{2} G_{2}(x) + \partial G_{1}(x) + G_{0}(x)$ belongs to the module $\mathcal{D}(\tilde{W},W)$.
\end{thm}
\begin{proof}
If $\mathcal V$  belongs to $\mathcal D(W, \tilde W)$ then, by Proposition \ref{basicprop}, 
\begin{equation*}
V:=\begin{pmatrix}
0&\mathcal V\\
0&0
\end{pmatrix}\in \mathcal D(\hat W), \quad \text{ for }  \quad \hat W=W\oplus\tilde W .
\end{equation*}

Thus $V^\dagger $ belongs to $\mathcal D(\hat W)$ and 
$V^\dag=
\begin{pmatrix}
0&0\\
\mathcal N&0
\end{pmatrix}
$ 
for some order-two linear differential operator $\mathcal N= \partial^{2} G_{2}(x) + \partial G_{1}(x) + G_{0}(x),$  with 
$G_{i}$  a polynomial of order less than or equal to $i$, for $i=0,1,2$.
Now we consider the differential operator 
$$ D:=V+V^\dag = \begin{pmatrix}
0&\mathcal V\\
\mathcal N&0
\end{pmatrix} ,
$$
symmetric with respect to $\hat W$.  
Then, it is a simple matter of careful integration by parts to see that the condition of symmetry for $D$ implies to the three differential equations \eqref{eq_sym} and the boundary conditions \eqref{cond_sym}.

Conversely, let us now assume that there are polynomials $G_0,G_1,G_2,$ with $\deg(G_{i})\leq i$, satisfying \eqref{eq_sym} and  \eqref{cond_sym}. We can define the differential operator 
$$\mathcal{N} := \partial^{2} G_{2}(x) + \partial G_{1}(x) + G_{0}(x),$$
and then the operator $$D:= \begin{pmatrix}
0&\mathcal V\\
\mathcal N&0
\end{pmatrix}=\partial^{2} D_{2}(x) + \partial D_{1}(x) + D_{0}(x)$$
is symmetric with respect to $\hat W$ due to \eqref{eq_sym} and  \eqref{cond_sym}.
Even more, since $D $ is symmetric and $\deg(D_{i})\leq i$ we have that it belongs to the algebra $\mathcal D(W)$. This means that 
$$
\begin{pmatrix}
P_n&0\\
0&\tilde P_n
\end{pmatrix} \cdot
\begin{pmatrix}
0&\mathcal V\\
\mathcal N&0
\end{pmatrix}=
\Lambda_n(D)
\begin{pmatrix}
P_n&0\\
0&\tilde P_n
\end{pmatrix}, \quad \text{ for all } n\in \mathbb{N}_{0},
$$
for some sequence of eigenvalues $\Lambda_n(D).$ It follows immediately that $\Lambda_n(D)$ has to be of the form $\begin{pmatrix}
0&A_n\\
B_n&0
\end{pmatrix}$. In particular, one has that 
$$P_n \cdot \mathcal V=A_n \tilde P_n, \quad \text{for all } n\in \mathbb{N}_{0},$$
proving that $\mathcal V\in\mathcal D (W,\tilde W)$ and
finishing the demonstration.
\end{proof}

\section{Darboux transformation of classical polynomials} \label{S-DTcp}

In this section, we give a criterion to decide whether a sequence of matrix-valued orthogonal polynomials is Darboux reducible to scalar orthogonal polynomials. We apply it to different examples. This criterion can be seen from two different but equivalent angles, below they are stated as Theorems \ref{T1} and \ref{T2}.

\begin{thm}\label{T1} Let  $\{P_n\}_{n\ge0}$ be a  sequence of matrix orthogonal polynomials of size $N$ 
and let $\{p_{j,n}\}_{n\ge0}$ be a scalar-valued sequence of orthogonal polynomials 
 for $j=1,\dots,N$.

 $P_n$  is a Darboux transformation of $p_{1,n} \oplus \cdots \oplus p_{N,n}$ if and only if for $j = 1 , \ldots, N$ there exists a $N\times 1$ vector-valued linear differential operator $D_{j}$ such that 
\begin{equation}\label{e1}
P_n \cdot D_j=\Lambda_{n}(D_{j})\,\,p_{j,n}, 
\end{equation}
and the $N\times 1$ vectors $\{\Lambda_{n}(D_{1}),\dots,\Lambda_{n}(D_{N})\}$ are  linearly independent for some  $n$.
\end{thm}
\begin{remark}\label{DinD(w,W)}
    Let us observe that the differential operator $D_j$ in \eqref{e1} belongs to the module $\mathcal D(W, w_j)$.
\end{remark}
\begin{proof}
By Theorem \ref{darboux-thm}, $P_n$  is a Darboux transformation of $\tilde P_{n}:=p_{1,n} \oplus \cdots \oplus p_{N,n}$ if and only if there is a linear differential operator $\mathcal V$
such that
$$ P_{n}(x) \cdot \mathcal{V}  = A_{n}\tilde {P}_{n}(x)$$
for $A_{n} \in \operatorname{Mat}_{N}(\mathbb{C})$ invertible for some $n$. If, for $j=1,\dots,N$, we define the operator $D_j$ as the $j$-th column of the operator $\mathcal{V}$, we have that the equation above is equivalent to have
$$
P_n \cdot D_j=\Lambda_{n}(D_{j})\,\,p_{j,n}, \quad \text{ for } j=1,\dots,N,
$$
where $\Lambda_{n}(D_{j})$ is the $j$-th column of $A_n$. Clearly $\{\Lambda_{n}(D_{1}),\dots,\Lambda_{n}(D_{N})\}$ is linearly independent if and only if $A_n$ is invertible. The theorem is proved.
\end{proof}

Naturally, we also have a dual version of the result above. The proof is completely analogous to the one of Theorem \ref{T1}.

\begin{thm}\label{T2} Let  $\{P_n\}_{n\ge0}$ be a $N\times N$ sequence of matrix orthogonal polynomials 
and let $\{p_{j,n}\}_{n\ge0}$ be a scalar-valued sequence of orthogonal polynomials 
 for $j=1,\dots,N$.

 $P_n$  is a Darboux transformation of $p_{1,n} \oplus \cdots \oplus p_{N,n}$ if and only if for $j = 1 , \ldots N$ there exist a $1\times N$ vector-valued linear differential operator $D_{j}$ such that 
\begin{equation}\label{e2}
\Lambda_{n}(D_{j})\,\,P_n =p_{j,n}\cdot D_j,
\end{equation}
with  the $1\times N$ vectors $\{\Lambda_{n}(D_{1}),\dots,\Lambda_{n}(D_{N})\}$ are  linearly independent for some $n$.
\end{thm}

\

\subsection{Example 1. Darboux transformation of a direct sum of classical scalar weights}\label{E1}
\mbox{}

Here we consider an example of matrix-valued orthogonal polynomials arising from matrix-valued spherical functions associated with the three-dimensional sphere studied in \cite{PTZ14}. We apply our Theorem \ref{T2} to prove that it is a Darboux transformation of a diagonal of scalar classic orthogonal polynomials and to build the full Darboux transformation explicitly.

Let $P_{n}(x)$ be the sequence of monic orthogonal polynomials with respect to 
\begin{equation*}
    W(x)= (1-x^2)^\frac 12\begin{pmatrix}
        3 & 3x& 4x^2-1\\ 3x& x^2+2 &3x \\4x^2-1 & 3x & 3
    \end{pmatrix},\quad x\in(-1,1).
\end{equation*}
We want to prove that $P_{n}(x)$ is a Darboux transformation of a direct sum of scalar-valued sequences of orthogonal polynomials. 
Namely, we want to find sequences of scalar polynomials $p_{j,n}$, orthogonal with respect to a weight $w(x)$, for which there is a $1\times 3$ second-order differential operator $D=\begin{pmatrix}d_1 & d_2 & d_3\end{pmatrix}$ satisfying \eqref{e2}, i.e., we need $D\in\mathcal D(w, W)$.

Let us consider a generic scalar Jacobi weight $w(x) = (1-x)^{\alpha}(1+x)^{\beta}$. By using Theorem \ref{op mod}, after some computations (solving the differential equations  \eqref{eq_sym} and the conditions \eqref{cond_sym}), we find that $\mathcal D(w, W)$ contains order one operators for $\alpha = \beta = \frac{3}{2}$. More precisely, these operators are linear combinations of 
\begin{equation*}
        {D}_{1} = \begin{pmatrix} \partial x  & -\partial 2  & \partial x + 4     \end{pmatrix},
 \quad 
{D}_{2} = \begin{pmatrix} -\partial - 1 & \partial 2x + 2  & -\partial  + 1 \end{pmatrix}  
\quad
\text{ and } 
\quad
    {D}_{3}  = \begin{pmatrix}  -1 & 0 & 1   \end{pmatrix}.
\end{equation*}

If $p_{n}(x)$ are the monic Jacobi polynomials associated to the weight $w(x) = (1-x^{2})^{\frac{3}{2}}$ then 
\begin{equation*}
    \begin{split}
        p_{n}(x) \cdot D_{1} & = \begin{pmatrix} n & 0 & n + 4 \end{pmatrix} P_{n}(x), \\
        p_{n}(x) \cdot D_{2} & = \begin{pmatrix} -1 & 2n+2 & 1 \end{pmatrix} P_{n}(x), 
        \\
        p_{n}(x) \cdot D_{3} & = \begin{pmatrix} -1 & 0 & 1 \end{pmatrix} P_{n}(x).
    \end{split}
\end{equation*}

Therefore, the set $\{ \Lambda_{n}(D_{1}), \Lambda_{n}(D_{2}), \Lambda_{n}(D_{3}) \}$ is linearly independent for all $n \in \mathbb{N}_{0}$. By Theorem \ref{T2} it follows that $P_n$ is a Darboux transformation of the direct sum of scalar Jacobi classical polynomials 
$\begin{psmallmatrix}p_n&0&0\\0&p_n&0\\0&0&p_n\end{psmallmatrix}$. Explicitly, we have that 
$ \mathcal V=\partial\begin{psmallmatrix}
    x&-2&x\\ -1&2x&-1\\0&0&0
\end{psmallmatrix} + 
\begin{psmallmatrix}
    0&0&4\\-1&2&1 \\ -1&0&1
\end{psmallmatrix}$ satisfies 
$$\begin{psmallmatrix}p_n&0&0\\0&p_n&0\\0&0&p_n\end{psmallmatrix}\cdot \mathcal V= A_nP_n, \qquad \text{ 
with } \quad 
A_n= \begin{psmallmatrix} n & 0 & n + 4 \\-1 & 2n+2 & 1\\-1 & 0 & 1   \end{psmallmatrix}.$$

\begin{remark}
    Notice that without having an 
    explicit expression for the orthogonal polynomials $P_{n}$ with respect to $W$, we were able to demonstrate, through the study of the modules $D(w, W)$, that $P_n$ is a Darboux transformation of a direct sum of Jacobi polynomials. 
   Even more,  
   we obtain an explicit expression for the orthogonal polynomials $P_{n}$ in terms of the polynomials $p_{n}$ by using the Darboux transformation. Explicitly, a sequence of orthogonal polynomials for $W$ is given by 
    $$p_{n}(x) \cdot \begin{pmatrix}D_{1} \\ D_{2} \\ D_{3} \end{pmatrix} = \begin{pmatrix} xp_{n}(x)' & -2p_{n}(x)' & x p_{n}(x)' + 4p_{n}(x) \\ -p_{n}(x)' - p_{n}(x) & 2x p_{n}(x)' + 2p_{n}(x) & -p_{n}(x)' + p_{n}(x) \\ -p_{n}(x) & 0 & p_{n}(x)\end{pmatrix}.$$
\end{remark}

\

\subsection{Example 2. A case that is not a Darboux transformation of scalar classical polynomials}\label{ex 2}
Let $a, \, b \in \mathbb{R} - \{0\}$. We consider the weight 
$$W(x) = e^{-x^{2}}\begin{pmatrix} a^{2}x^{2} + e^{2bx} & ax \\ ax & 1 \end{pmatrix}.$$
In \cite{BP23-1}, the authors provided an explicit expression of a sequence of orthogonal polynomials for $W$,
$$Q_{n}(x) = \begin{pmatrix}H_{n}(x-b) && aH_{n+1}(x)-axH_{n}(x-b) \\ -anH_{n-1}(x-b) && 2e^{b^{2}}H_{n}(x)+a^{2}nxH_{n-1}(x-b)\end{pmatrix},$$
where $H_{n}(x)$ is the sequence of monic orthogonal polynomials for the scalar Hermite weight $e^{-x^{2}}$, and they proved that the algebra $\mathcal{D}(W)$ is a polynomial algebra over the second-order differential operator
\begin{equation*}\label{op Dd}
D = \partial^{2}I + \partial \begin{pmatrix} -2x + 2b && -2abx + 2a \\
0 && -2x\end{pmatrix} + \begin{pmatrix} -2 && 0 \\ 0 && 0 \end{pmatrix}.
\end{equation*}

As a consequence, $Q_n$ is not a Darboux transformation of scalar orthogonal polynomials.
Here, in a more direct manner, we prove a stronger statement: For any scalar weight $w$ supported on $\mathbb{R}$, the modules 
$\mathcal{D}(W,w)$ and   $\mathcal{D}(w, W)$ are zero (see  Theorem \ref{T1} and Remark \ref{DinD(w,W)}). 

\begin{prop}
    The modules $\mathcal{D}(w, W)$ and $\mathcal{D}( W,w)$ are zero for any scalar weight $w$  on $\mathbb{R}$.
\end{prop}
\begin{proof}
    Let $w(x)$ be a scalar weight supported on $\mathbb{R}$. Let $D_{1} = \begin{pmatrix} d_{1} & d_{2} \end{pmatrix} $ be a differential operator in the module $\mathcal{D}(w,W)$. By Proposition \ref{basicprop}, it follows that $D_{2} = W(x)D_{1}^{\ast}w(x)^{-1} = \begin{pmatrix} e^{-x^{2}}(a^{2}x^{2}d_{1}^{\ast}+ axd_{2}^{\ast})w(x)^{-1} + e^{-x^{2}+2bx}d_{1}^{\ast}w(x)^{-1} \\ e^{-x^{2}}(axd_{1}^{\ast}+ d_{2}^{\ast})w(x)^{-1} \end{pmatrix}$ belongs to $\mathcal{D}(W,w)$.
    Assume that $e^{-x^{2}}w(x)^{-1}$ is a rational function. Therefore, we have that $e^{-x^{2}+2bx}w(x)^{-1}$ is not a rational function. Since $D_{2}$ must have polynomial coefficients, it follows that $d_{1} = 0$. From here, we obtain that $D_{2} = \begin{pmatrix} e^{-x^{2}}axd_{2}^{\ast}w(x)^{-1} \\ e^{-x^{2}}d_{2}^{\ast}w(x)^{-1} \end{pmatrix}$. By \eqref{D(W,w)D(w,W)}, we have that 
    $$D_{2}D_{1} = \begin{pmatrix} 0 & e^{-x^{2}}axd_{2}^{\ast}w(x)^{-1}d_{2} \\ 0 & e^{-x^{2}}d_{2}^{\ast}w(x)^{-1}d_{2} \end{pmatrix} \in \mathcal{D}(W).$$
By straightforward computations, it can be checked that the only operator with its first column zero that has $Q_{n}$ as eigenfunction, for all $n\in \mathbb{N}_{0}$, is the operator zero. Thus, $D_{1} = 0$.

\noindent
If $e^{-x^{2}}w(x)^{-1}$ is not a rational function, then $D_{2} = \begin{pmatrix} e^{-x{2}+2bx}d_{1}^{\ast}w(x)^{-1} \\ 0 \end{pmatrix}$ and we proceed in an analogous way. Hence, $\mathcal{D}(w,W)$ is zero.

    Finally, from Proposition \ref{basicprop} we also obtain that $\mathcal{D}(W, w)=0$.
\end{proof}

\section{Darboux transformation of matrix-valued polynomials}\label{S-DT}

In this final section, we go beyond Darboux reducibility to scalar polynomials. 
 One of the motivations for this is that a sequence of polynomials that is not a Darboux transformation of classical polynomials could still be Darboux-reducible. For a better understanding of this, we will need the following theorems, which are non-scalar versions of Theorems \ref{T1} and \ref{T2}.

\begin{thm}\label{T3} Let  $\{P_n\}_{n\ge0}$ be a $N\times N$ sequence of matrix-valued orthogonal polynomials 
and let $\{p_{j,n}\}_{n\ge0}$ be a $d_j\times d_j$ sequence of matrix-valued orthogonal polynomials 
 for $j=1,\dots,M$.

 $P_n$  is a Darboux transformation of $ p_{1,n} \oplus \cdots \oplus p_{M,n} $ if and only if, for each $j = 1 , \ldots , M$, there exists a $ N \times d_j$ matrix-valued linear differential operator $D_{j}$ satisfying that 
$$
\Lambda_{n}(D_{j})\,\,p_{j,n} =P_n \cdot D_j\quad \text{ for all } n\in \mathbb{N}_{0},
$$
for some $ N \times d_j$ matrix $\Lambda_n(D_j)$, and the $N\times N$ matrix given by
$$
\begin{pmatrix}\|&\|&&\|\\  
\Lambda_{n}(D_{1})&\Lambda_{n}(D_{2}) &\dots&\Lambda_{n}(D_{M})\\
\|&\|&&\| 
   \end{pmatrix}$$
is invertible for some $n$.
\end{thm}
\begin{proof}
This proof is almost the same as the one of Theorem \ref{T1}, we add it just for completeness.
By Theorem \ref{darboux-thm}, $P_n$  is a Darboux transformation of $\tilde P_{n}:=p_{1,n} \oplus \cdots \oplus p_{M,n}$ if and only if there is a linear differential operator $\mathcal V$
such that
$$ P_{n}(x) \cdot \mathcal{V}  = A_{n}\tilde {P}_{n}(x)$$
for $A_{n} \in \operatorname{Mat}_{N}(\mathbb{C})$ invertible for some $n$. For $j=1,\dots,M$, we define the operator $D_j$ as the $N\times d_j$ operator consisting of the columns of the operator $\mathcal{V}$ from the $(d_1+\dots+d_{j-1}+1)$-th column to the $(d_1+\dots+d_{j})$-th column. Then we have that the equation above is equivalent to have
$$
P_n \cdot D_j=\Lambda_{n}(D_{j})\,\,p_{j,n}, \quad \text{ for } j=1,\dots,M,
$$
where $\Lambda_{n}(D_{j})$ is a $N\times d_j$ matrix consisting of the columns of $A_n$ from the $(d_1+\dots+d_{j-1}+1)$-th column to the $(d_1+\dots+d_{j})$-th column.  
The theorem is proved.
\end{proof}

The following theorem is a dual version of Theorem \ref{T3} and its proof is completely analogous.
\begin{thm}\label{T3 dual} Let  $\{P_n\}_{n\ge0}$ be a $N\times N$ sequence of matrix-valued orthogonal polynomials
and let $\{p_{j,n}\}_{n\ge0}$ be a $d_j\times d_j$ sequence of matrix-valued orthogonal polynomials 
 for $j=1,\dots,M$.

 $P_n$  is a Darboux transformation of $ p_{1,n} \oplus \cdots \oplus p_{M,n} $ if and only if, for each $j = 1 , \ldots , M$, there exists a $d_j\times N$ matrix-valued linear differential operator $D_{j}$ satisfying that
$$
\Lambda_{n}(D_{j})\,\,P_n =p_{j,n}\cdot D_j
$$
for some $d_j\times N$ matrix  $\Lambda_n(D_j)$, and the  $N\times N$ matrix given by
and the $N\times N$ matrix given by
$$
\begin{pmatrix} 
=\joinrel=\joinrel=
 \Lambda_{n}(D_{1}) 
=\joinrel=\joinrel= \\  
=\joinrel=\joinrel=\Lambda_{n}(D_{2}) 
=\joinrel=\joinrel= \\ \vdots \\  
=\joinrel=\joinrel=\Lambda_{n}(D_{M}) 
=\joinrel=\joinrel=  \end{pmatrix}$$
is invertible for some $n$.
\end{thm}

\medskip

\subsection{Example 3. Darboux-reducible to a direct sum of $2\times 2$ orthogonal polynomials.}\label{E3}
In this section, we analyze a sequence of matrix-valued orthogonal polynomials of size $4$ that happen to be Darboux-equivalent to a direct sum of two sequences of matrix-valued orthogonal polynomials of size $2$.
Let us consider the following Laguerre-type weight from \cite{B25}
$$W(x) = e^{-x}x^{\alpha} \begin{pmatrix}x^{2} + x && x && x^{2} && 0 \\ x && 1 && x && 0 \\ x^{2} && x && 2x^{2} + 1 && x \\ 0 && 0 && x && 1 \end{pmatrix}.$$
A sequence of orthogonal polynomials for $W$ is given by 
$$Q_{n}(x) = \begin{psmallmatrix} \ell_{n}^{(\alpha+1)}(x) & -(n+\alpha+1)\ell_{n}^{(\alpha)}(x) & 0 & 0 \\ -n \ell_{n-1}^{(\alpha+1)}(x) & n\ell_{n-1}^{(\alpha+1)}(x)x+\ell_{n}^{(\alpha)}(x)+n(n+\alpha)\ell_{n-1}^{(\alpha)}(x)x & - n (n+\alpha)\ell_{n-1}^{(\alpha)}(x) & n(n+\alpha) \ell_{n-1}^{(\alpha)}(x)x \\ 0 & \ell_{n+1}^{(\alpha)}(x)-\ell_{n}^{(\alpha)}(x)x & \ell_{n}^{(\alpha)}(x) & \ell_{n+1}^{(\alpha)}(x)-\ell_{n}^{(\alpha)}(x)x \\ 0 & n (n+\alpha)\ell_{n-1}^{(\alpha)}(x)x & -n(n+\alpha)\ell_{n-1}^{(\alpha)}(x) & n(n+\alpha)\ell_{n-1}^{(\alpha)}(x)x + \ell_{n}^{(\alpha)}(x) \end{psmallmatrix},$$
where $\ell_{n}^{(\alpha)}$ is the sequence of monic orthogonal polynomials for the scalar Laguerre weight $e^{-x}x^{\alpha}$.

Let $w_{1}(x)$ and $w_{2}(x)$ be the $2 \times 2$ irreducible weight matrices that appear in  \cite{CMV07} and \cite{DG04}, respectively,
$$w_{1}(x) = e^{-x}x^{\alpha} \begin{pmatrix} x(1+x) && x \\ x && 1 \end{pmatrix}, \quad w_{2}(x) = e^{-x}x^{\alpha} \begin{pmatrix} 1 + x^{2} && x \\x && 1 \end{pmatrix}.$$

We have that 
$$P_{1,n}(x) = \begin{psmallmatrix} \ell_{n}^{(\alpha+1)} & -(n+\alpha+1)\ell_{n}^{(\alpha)} \\ -n\ell_{n-1}^{(\alpha+1)} & n\ell_{n-1}^{(\alpha+1)}(x)x + \ell_{n}^{(\alpha)}(x) \end{psmallmatrix} \text{ and } P_{2,n}(x) = \begin{psmallmatrix} \ell_{n}^{(\alpha)}(x) & \ell_{n+1}^{(\alpha)}(x)-\ell_{n}^{(\alpha)}(x)x \\ - n(n+\alpha)\ell_{n-1}^{(\alpha)}(x) & n(n+\alpha)\ell_{n-1}^{(\alpha)}(x)x+\ell_{n}^{(\alpha)}(x)\end{psmallmatrix}$$
are sequence of orthogonal polynomials for $w_{1}$ and $w_{2}$, respectively.

From the explicit expression of the orthogonal polynomials, we find that the differential operators 
\begin{align*}
        D_{1} & = \partial^{3} \begin{psmallmatrix} 0 &&  x^{3} &&  -x^{2} && x^{3} \\ 0 && 0 && 0 && 0 \end{psmallmatrix}  + \partial^{2} \begin{psmallmatrix}-x  & 2x^{2}(\alpha+2) &  -(2 \alpha+2)x & 2x^{2}(\alpha+2) \\  0 &  x(x-1) &  -x &  x^{2} \end{psmallmatrix} \\
        & \quad + \partial \begin{psmallmatrix} x-\alpha-2 & x(\alpha^{2}+3 \alpha+1) &  -(\alpha+2)(\alpha+1) &  (\alpha+2)(\alpha+1)x \\  0 &   \alpha x -\alpha+2x-1 & -\alpha-1 & (\alpha+1)x \end{psmallmatrix} + \begin{psmallmatrix}  2 && - \alpha-1 &&  0 && 0 \\ 0 && 1 && 0 && 0\end{psmallmatrix}
    \intertext{and}
        D_{2} & = \partial^{4} \begin{psmallmatrix} 0 &&  0 &&  x^{2} &&  0 \\ 0 && x^{2} && 0 &&  x^{2} \end{psmallmatrix} + \partial^{3} \begin{psmallmatrix} 0 &  4x^{2} &  2x(\alpha+2-x) &  4x^{2} \\ 0 &  2x(+\alpha+2-x) &  0 &  2x(\alpha+2-x) \end{psmallmatrix} \\
        & \quad + \partial^{2} \begin{psmallmatrix} 0 &  x(6\alpha+13-4x) & \alpha (\alpha-3x)+x^{2}+3 \alpha-7x+2 & 2x(3\alpha+6-2x) \\  0 &  \alpha (\alpha-3x)+x^{2}+3 \alpha-5x+2 & 0 & \alpha^{2}-3 \alpha x+x^{2}+3 \alpha-5x+2 \end{psmallmatrix} \\
        & \quad  + \partial \begin{psmallmatrix}   0 &  (2 \alpha+5)(\alpha+1-2x) &  -(\alpha+3)(\alpha+1-x) & (2\alpha+4)(\alpha+1-2x) \\ 0 & -(\alpha+1)(\alpha+1-x) & 0 & -(\alpha+1)(\alpha+1-x)\end{psmallmatrix} \\ 
        & \quad + \begin{psmallmatrix} 0 & -( \alpha+2)(\alpha+1) & \alpha+2 & -(\alpha+1)^{2}\\  0 & 0 &  0 &  1\end{psmallmatrix}
    \end{align*}
satisfy
\begin{equation*}
    \begin{split}
        P_{1,n}(x) \cdot D_{1} & =\Lambda_n(D_1)Q_n(x)= \begin{pmatrix} n +2 & 0 & 0 & 0 \\ 0 & n + 1 & 0 & 0  \end{pmatrix} Q_n(x) ,\\
        P_{2,n}(x) \cdot D_{2} & = \Lambda_n(D_2)Q_n(x)=\begin{pmatrix} 0 & 0 & (n+1)(\alpha + n + 1) + 1 & 0 \\ 0 & 0 & 0  & n(n+\alpha) + 1 \end{pmatrix} Q_n(x).
    \end{split}
\end{equation*}
In other words, $D_{1}$ belongs to $\mathcal{D}(w_{1},W)$  and $D_{2}$ belongs to $\mathcal{D}(w_{2},W)$.
On the other hand, the matrix $$\begin{pmatrix}
=\joinrel=\joinrel=\Lambda_{n}(D_{1}) 
=\joinrel=\joinrel= \\  
=\joinrel=\joinrel=\Lambda_{n}(D_{2})
=\joinrel=\joinrel=\end{pmatrix} = \begin{psmallmatrix} n + 2 & 0 & 0 & 0 \\ 0 & n+1 & 0 & 0 \\ 0 & 0 & (n+1)(\alpha+n+1)+1 & 0 \\ 0 & 0 & 0 & n(n+\alpha)+1 \end{psmallmatrix}$$ is an invertible matrix, for all $n\in \mathbb{N}_{0}$.  Therefore, by Theorem \ref{T3 dual}, we have that $Q_{n}(x)$ is a Darboux transformation of the direct sum $P_{1,n}(x) \oplus P_{2,n}(x)$.

\medskip

\subsection{Example 4. An example of size $3$ that is not Darboux-reducible to scalars and $\mathcal D(W,w)\neq0$} 
Let us consider the following Hermite-type weight from \cite{B25}
\begin{equation*}\label{weight}
W(x) = e^{-x^{2}}\begin{pmatrix} e^{2bx} + a^{2}x^{2} && ax && acx^{2} \\ ax && 1 && cx \\ acx^{2} && cx && c^{2}x^{2} + 1  \end{pmatrix},
\end{equation*}
with $a,b,c \in \mathbb{R}-\{0\}$. A sequence of orthogonal polynomials for $W$ is given by 
$$Q_{n}(x) = \begin{psmallmatrix} H_{n}(x-b) && a\big(H_{n+1}(x)-xH_{n}(x-b)\big) && 0 \\ -\frac{n ae^{-b^{2}}}{2}H_{n-1}(x-b) && \frac{n a^{2}e^{-b^{2}}}{2}\, x  H_{n-1}(x-b) + H_{n}(x) + \frac{nc^{2}}{2}\, xH_{n-1}(x) && - \frac{cn}{2}H_{n-1}(x)\\ 0 && -cnH_{n-1}(x) && 2H_{n}(x)\end{psmallmatrix},$$
where $H_{n}(x)$ is the sequence of monic orthogonal polynomials for the scalar Hermite weight $w(x)=e^{-x^{2}}$.
 
 The differential operators 
\begin{equation*}
\mathcal{D}_{1} = \partial^{2}I + \partial \begin{pmatrix}  2b-2x && -2bax+2a &&  0 \\ 0 && -2x && 0 \\ 0 && 2c && -2x \end{pmatrix} + \begin{pmatrix} 0 && 0 && 0 \\ 0 && 2 && 0 \\ 0 && 0 && 0 \end{pmatrix},
\end{equation*}
and
\begin{equation*}
    \mathcal{D}_{2} = \partial^{2} \begin{pmatrix} 0 && ax && 0 \\ 0 && 1 && 0 \\ 0 && cx && 0 \end{pmatrix} + \partial \begin{pmatrix}  0 && 2a &&  -\frac{2ax}{c} \\ 0 && 0 &&  -\frac{2}{c} \\ 0 &&  2c+\frac{2}{c} && -2x \end{pmatrix} + \begin{pmatrix} 2+\frac{4}{c^{2}} &&  0 && -\frac{2a}{c} \\  0 &&  2+\frac{4}{c^{2}} &&   0 \\  0 && 0 && 0 \end{pmatrix}
\end{equation*}
are $W$-symmetric operators in the algebra $\mathcal D(W)$. 
Its eigenvalues 
for the monic sequence of orthogonal polynomials for $W$ are given by 
$$\Lambda_{n}(\mathcal{D}_{1}) = \begin{pmatrix} -2n && -2ban && 0 \\ 0 && -2n + 2 && 0 \\ 0 && 0 && -2n\end{pmatrix}, \quad \Lambda_{n}(\mathcal{D}_{2}) = \begin{pmatrix}2 + \frac{4}{c^{2}} && 0 && -\frac{2a}{c}(n+1) \\ 0 && 2+ \frac{4}{c^{2}} && 0 \\ 0 && 0 && -2n \end{pmatrix}.$$
We have computational evidence that the algebra $\mathcal{D}(W)$ is generated by $\{\mathcal{D}_{1}, \mathcal{D}_{2}, I\}$.  

\smallskip
Now, we study the modules $\mathcal{D}(w, W)$ for a scalar weight $w(x)$ supported on $\mathbb{R}$. If $e^{-x^{2}}w(x)^{-1}$ and $e^{-x^{2}+2bx}w(x)^{-1}$ are not rational function, then $\mathcal{D}(w,W)=\{0\}$. In fact, for  $D_{1} = \begin{pmatrix} d_{1} & d_{2} & d_{3} \end{pmatrix} \in \mathcal{D}(w,W)$ and following the proof of Theorem \ref{dr}, we obtain that the differential operator $\tilde D_{1} =W(x)D_{1}^{\ast}w(x)^{-1}$ belongs to the module $\mathcal{D}(W,w)$. We have 
$$\tilde{D}_{1} = \begin{pmatrix}e^{-x^{2}}(a^{2}x^{2}d_{1}^{\ast} + ax^{2}cd_{3}^{\ast} + axd_{2}^{\ast})w(x)^{-1} + e^{-x^{2}+2bx}d_{1}^{\ast}w(x)^{-1} \\ e^{-x^{2}}(axd_{1}^{\ast} + cxd_{3}^{\ast}+ d_{2}^{\ast})w(x)^{-1} \\ e^{-x^{2}}(acx^{2}d_{1}^{\ast} + c^{2}x^{2}d_{3}^{\ast} + cxd_{2}^{\ast} + d_{3}^{\ast})w(x)^{-1} \end{pmatrix}.$$
Since the coefficients of the differential operator $\tilde{D}_{1}$ must be polynomials, we obtain that $\tilde{D}_{1} = 0$. Thus, $D_{1} = 0$.
In conclusion, we only have to consider the classical scalar Hermite weights $w(x) = e^{-x^{2}}$, or $w(x) =e^{-x^{2}+2bx}$.

Based on strong computational evidence, we conjectured that 
\(\mathcal{D}(w, W) = 0\), for 
 the weight \(w(x) = e^{-x^{2}+2bx}\). (Up to order 10, we have not found any differential operator in \(\mathcal{D}(w, W)\)).

\smallskip

For the weight $w(x) = e^{-x^{2}}$, we have that the last row of $Q_{n}(x)$ is
$$\begin{pmatrix} 0 & 0 & 1 \end{pmatrix} Q_{n}(x) = \begin{pmatrix} 0 & -cnH_{n-1}(x) & 2H_{n}(x) \end{pmatrix}=H_{n}(x) \cdot \begin{pmatrix} 0 & -\partial \, c & 2 \end{pmatrix}.$$ 
Hence, the differential operator 
$$D_{1} = \partial \begin{pmatrix} 0 & -c & 0 \end{pmatrix} + \begin{pmatrix} 0 & 0 & 2 \end{pmatrix} \in \mathcal{D}(w,W)$$
and $$\Lambda_n(D_1)= \begin{pmatrix} 0 & 0 & 1 \end{pmatrix} .$$

All operators up to order $10$ obtained in $\mathcal{D}(w, W)$ satisfy that their eigenvalue is linearly dependent on the eigenvalue of $D_{1}$. 

Besides, one can find other differential operators in the module \(D(w, W)\) by multiplying \(D_{1}\) on the right by any operator in the algebra \(D(W)\), or by multiplying it on the left by any operator in the algebra \(\mathcal{D}(w)\). However, in both cases, the eigenvalue of the new operator takes the form \(p(n) \begin{pmatrix} 0 & 0 & 1 \end{pmatrix}\) for a polynomial \(p\) which is linearly dependent on the eigenvalue of \(D_{1}\). 

In any case, we proved that there are no three differential operators up to order $10$ that satisfy the conditions in Theorem \ref{T2}. Therefore, we conjecture that $W$ is not a Darboux transformation of a direct sum of classical weights, but it is Darboux-reducible. Namely, it is Darboux-equivalent to $w\oplus \tilde W$, for some $\tilde W$ weight matrix of size $2$ which is not Darboux-equivalent to a direct sum of scalar weights.

\

\section*{Acknowledgements}
The research of I. Bono Parisi was supported by SeCyT-UNC, CONICET, PIP 1220150100356 and by AUIP and Junta de Andalucía PMA2-2023-056-14. The research of I. Pacharoni was supported by SeCyT-UNC, CONICET, PIP 1220150100356.
The research of I. Zurri\'an was supported by   Ministerio de Ciencia e Innovaci\'on PID2021-124332NB-C21 and Universidad de Sevilla VI PPIT - US.
Disclosure statement: there are no competing interests to declare.

\bibliographystyle{a9}

\end{document}